\documentclass[11pt]{amsart}

\usepackage{lmodern}
\usepackage{microtype}
\usepackage[english]{babel}
\usepackage{mathtools}
\usepackage{hyperref}
\usepackage[msc-links]{amsrefs}

\theoremstyle{plain} 
\newtheorem*{theorem}{Theorem}
\newtheorem*{lemma}{Lemma}
\newtheorem*{corollary}{Corollary}

\begin{document} 
\title[The Riesz projection is bounded on $L^p(\mathbb{T})$ for $1<p<\infty$]{A simple proof that the Riesz projection is bounded on $L^p(\mathbb{T})$ for $1<p<\infty$}

\begin{abstract}
	Let $\mathbf{P}$ denote the Riesz projection on the unit circle $\mathbb{T}$ and suppose that $1<p<\infty$. We present a simple proof of the bound $\|\mathbf{P}f\|_p \leq \max(p,q) \|f\|_p$, where $f$ is in $L^p(\mathbb{T})$ and $p^{-1}+q^{-1}=1$. Our proof is a variation of a classical argument due to M.~Riesz demonstrating that the Hilbert transform is bounded on $L^p(\mathbb{T})$.
\end{abstract}

\date{\today}


\thanks{Research supported by Grant 354537 of the Research Council of Norway.}

\author{Ole Fredrik Brevig} 
\address{Department of Mathematical Sciences, Norwegian University of Science and Technology (NTNU), 7491 Trondheim, Norway} 
\email{ole.brevig@ntnu.no}

\maketitle

This note contains a simple proof of an important result in Fourier analysis. A well-known consequence of Fej\'{e}r's theorem is that the set of trigonometric polynomials is dense in $L^p(\mathbb{T})$ for $1 \leq p<\infty$, where $\mathbb{T}$ is the unit circle endowed with the normalized Lebesgue arc length measure. This implies that the Riesz projection $\mathbf{P}$ can be densely defined on $L^p(\mathbb{T})$ by 
\[\mathbf{P}\left(\sum_{n \in \mathbb{Z}} \widehat{f}(n)\, e^{in\theta}\right) = \sum_{n=0}^\infty \widehat{f}(n) \,e^{in\theta}.\]
The Hilbert transform $\mathbf{H}$ is the quintessential example of a singular integral operator, and it can be related to the Riesz projection via the formula
\begin{equation} \label{eq:conjfunc}
	\mathbf{H}f = -i\left(2\mathbf{P}f - f - \widehat{f}(0)\right).
\end{equation}
A celebrated result due to M.~Riesz~\cite{Riesz1928} is that $\mathbf{H}$ defines a bounded linear operator on $L^p(\mathbb{T})$ for $1<p<\infty$. We refer to Littlewood~\cite{Littlewood1986}*{pp.~194--195} and G{\aa}rding~\cite{Garding1970}*{pp.~III--IV} for historical context and to Grafakos~\cite{Grafakos2014}*{pp.~247--248} for a modern presentation of the original proof.

The crux of the present note lies in the observation that this proof can be made significantly shorter and more elegant by shifting the focus from $\mathbf{H}$ to $\mathbf{P}$, while following the exact same blueprint: establish the result first for even integers $p$, then extend to the general case by interpolation and duality. 

\begin{lemma}
	If $k=1,2,3,\ldots$ and if $f$ is a trigonometric polynomial, then 
	\[\|\mathbf{P}f\|_{2k} \leq k \|f\|_{2k}.\]
\end{lemma}

\begin{proof}
	Let $\mathbf{P}_\perp f = f-\mathbf{P}f$ and assume that $\|\mathbf{P}f\|_{2k} \geq \|\mathbf{P}_\perp f\|_{2k}$. Since $f$ is a trigonometric polynomial and $k$ is a positive integer, it is plain that
	\begin{equation} \label{eq:mstrick}
		(\mathbf{P}f)^k \perp (-\mathbf{P}_\perp f)^k
	\end{equation}
	in $L^2(\mathbb{T})$. Since $\|\mathbf{P}f\|_{2k}^k = \|(\mathbf{P}f)^k\|_2$, this can be parlayed into the estimate
	\begin{equation} \label{eq:step1}
		\left\|\mathbf{P}f\right\|_{2k}^k \leq \left\|(\mathbf{P}f)^k - (-\mathbf{P}_\perp f)^k\right\|_2 = \left\| f \sum_{j=0}^{k-1} (-1)^j (\mathbf{P} f)^{k-1-j}(\mathbf{P}_\perp f)^j 
\right\|_2.
	\end{equation}
	Using H\"older's inequality and the triangle inequality, we infer from \eqref{eq:step1} that
	\begin{equation} \label{eq:step2}
		\left\|\mathbf{P}f\right\|_{2k}^k \leq \left\|f\right\|_{2k} \sum_{j=0}^{k-1} \left\|(\mathbf{P}f)^{k-1-j} (\mathbf{P}_\perp f)^j \right\|_{\frac{2k}{k-1}}.
	\end{equation}
	More applications of H\"older's inequality demonstrate that the terms in this sum are each bounded by $\|\mathbf{P}f\|_{2k}^{k-1-j} \|\mathbf{P}_\perp f\|_{2k}^j \leq \left\|\mathbf{P}f\right\|_{2k}^{k-1}$. Hence 
	\[\left\|\mathbf{P}f\right\|_{2k}^k \leq k \left\|f\right\|_{2k} \left\|\mathbf{P}f\right\|_{2k}^{k-1},\]
	which implies the stated bound. If $\|\mathbf{P}f\|_{2k} \leq \|\mathbf{P}_\perp f\|_{2k}$, then the same argument yields that $\|\mathbf{P}f\|_{2k} \leq \|\mathbf{P}_\perp f\|_{2k} \leq k \|f\|_{2k}$. 
\end{proof}

The case $k=2$ of the main trick \eqref{eq:mstrick} in the above proof is from Marzo and Seip \cite{MS2011}*{Theorem~1}. The algebraic identity used in \eqref{eq:step1} plays a similar role in the proof of a result due to Forelli~\cite{Forelli1963}*{Lemma~4}, while our use of H\"older's inequality to handle the terms of the sum in \eqref{eq:step2} follows M.~Riesz~\cite{Riesz1928}. 

In what follows $q$ will be the conjugate exponent of $p$, so that $p^{-1}+q^{-1}=1$.
 
\begin{theorem}
	If $1<p<\infty$, then $\mathbf{P}$ extends to a bounded linear operator on $L^p(\mathbb{T})$ satisfying
	\[\left\|\mathbf{P}f\right\|_p \leq \max(p,q) \left\|f\right\|_p.\]
\end{theorem}

\begin{proof}
	If $k=1,2,3,\ldots$, then the lemma shows that $\mathbf{P}$ extends by density and continuity to a bounded linear operator on $L^{2k}(\mathbb{T})$ with norm at most $k$. If $2k \leq p \leq 2(k+1)$, then we use the Riesz-Thorin interpolation theorem to infer from this that $\mathbf{P}$ extends to a bounded linear operator on $L^p(\mathbb{T})$ with norm at most $k^{1-\theta} (k+1)^\theta$ for some $0 \leq \theta \leq 1$ that depends only on $p$. Since 
	\[k \leq k+1 \leq 2k \leq p,\]
	we obtain the stated estimate for the case $2 \leq p<\infty$. Since $\langle \mathbf{P}f,g\rangle = \langle f, \mathbf{P}g \rangle$ for every pair of trigonometric polynomials $f$ and $g$, the case $1<p<2$ can be established via duality.
\end{proof}

We have made no effort to optimize the constant $\max(p,q)$ in the above bound. It is possible to extract $C\max(p,q)$ for $C=1/(e\log{2})=0.5307\ldots$ from the same argument by a small calculus computation. The best constant is $1/\sin(\pi/p)$ by a result due to Hollenbeck and Verbitsky~\cite{HV2000}.

The following result is immediate from \eqref{eq:conjfunc} and the theorem.
\begin{corollary}[M.~Riesz~\cite{Riesz1928}]
	If $1<p<\infty$, then $\mathbf{H}$ extends to a bounded linear operator on $L^p(\mathbb{T})$ satisfying
	\[\|\mathbf{H}f\|_p \leq 2\left(\max(p,q)+1\right) \|f\|_p.\]
\end{corollary}
The best constant in this bound is due to Pichorides~\cite{Pichorides1972}. 
\bibliography{rieszprojection}

\end{document}